\documentclass[a4paper,10pt,reqno]{amsart}

\usepackage[utf8]{inputenc}
\usepackage[T1]{fontenc}
\usepackage{amsthm}
\usepackage{amsmath}
\usepackage{amssymb}
\usepackage[inline]{enumitem}
\usepackage{comment}
\PassOptionsToPackage{hyphens}{url}\usepackage{hyperref}
\usepackage{fancyhdr}
\usepackage{mathrsfs}
\usepackage{stmaryrd}
\usepackage[normalem]{ulem}
\usepackage{xcolor}

\setlist[description]{%
  itemsep=0.05cm,               % space between items
  font={\normalfont\textsc}, % set the label font
 leftmargin=\parindent,
 labelindent=\parindent
}

\theoremstyle{definition}

\newtheorem{theorem}{Theorem}[section]
\newtheorem{lemma}[theorem]{Lemma}

\newtheorem{corollary}[theorem]{Corollary}

\theoremstyle{definition}

\newtheorem{remark}[theorem]{Remark}

\newtheorem{examples}[theorem]{Examples}

\definecolor{blue-url}{RGB}{0,0,100}
\definecolor{red-url}{RGB}{100,0,0}
\definecolor{green-url}{RGB}{0,100,0}
\definecolor{light-yellow}{RGB}{255,255,128}
\definecolor{light-blue}{RGB}{193,255,255}
\definecolor{light-red}{RGB}{239,83,80}

\hypersetup{
	pdftitle={Power monoids},
	pdfauthor={John Doe},%Tringali},
	pdfmenubar=false,
	pdffitwindow=true,
	pdfstartview=FitH,
	colorlinks=true,
	linkcolor=blue-url,
	citecolor=green-url,
	urlcolor=red-url
}

\renewcommand{\emptyset}{\varnothing}
\renewcommand{\setminus}{\smallsetminus}
\renewcommand{\,}{\kern 0.1em}

\providecommand\llb{\llbracket}
\providecommand\rrb{\rrbracket}

\newcommand{\evid}[1]{\textsf{#1}}
\newcommand{\fin}{\mathrm{fin}}

{\newline\vspace{\abovedisplayskip}\hbox to \textwidth\bgroup\hss$\displaystyle}
{$\hss\egroup\vspace{\belowdisplayskip}}

\makeatletter
\DeclareFontFamily{OMX}{MnSymbolE}{}
\DeclareSymbolFont{MnLargeSymbols}{OMX}{MnSymbolE}{m}{n}
\SetSymbolFont{MnLargeSymbols}{bold}{OMX}{MnSymbolE}{b}{n}
\DeclareFontShape{OMX}{MnSymbolE}{m}{n}{
	<-6>  MnSymbolE5
	<6-7>  MnSymbolE6
	<7-8>  MnSymbolE7
	<8-9>  MnSymbolE8
	<9-10> MnSymbolE9
	<10-12> MnSymbolE10
	<12->   MnSymbolE12
}{}
\DeclareFontShape{OMX}{MnSymbolE}{b}{n}{
	<-6>  MnSymbolE-Bold5
	<6-7>  MnSymbolE-Bold6
	<7-8>  MnSymbolE-Bold7
	<8-9>  MnSymbolE-Bold8
	<9-10> MnSymbolE-Bold9
	<10-12> MnSymbolE-Bold10
	<12->   MnSymbolE-Bold12
}{}

\let\llangle\@undefined
\let\rrangle\@undefined
\DeclareMathDelimiter{\llangle}{\mathopen}%
{MnLargeSymbols}{'164}{MnLargeSymbols}{'164}
\DeclareMathDelimiter{\rrangle}{\mathclose}%
{MnLargeSymbols}{'171}{MnLargeSymbols}{'171}
\makeatother
\begin{document}
\title{A conjecture by Bienvenu and Geroldinger on power monoids}
\author{Salvatore Tringali}
\address{(S.~Tringali) School of Mathematical Sciences, Hebei Normal University | Shijiazhuang, Hebei province, 050024 China}
\email{salvo.tringali@gmail.com}
%\urladdr{http://imsc.uni-graz.at/tringali}
%
\author{Weihao Yan}
\address{(W.~Yao) School of Mathematical Sciences, Hebei Normal University | Shijiazhuang, Hebei province, 050024 China}
\email{weihao.yan.hebnu@outlook.com}
%
%\thanks{This manuscript was created on: \today{} at \currenttime}
%Secondary \ldots}
%
% 06F05: Ordered semigroups and monoids
% 11B13: Additive bases, including sumsets 
% 11B30: Arithmetic combinatorics; higher degree uniformity
% 13A05: Divisibility and factorizations in commutative rings
% 13F15: Commutative rings defined by factorization properties
% 13P05: Polynomials, factorization in commutative rings
% 15A23: Factorization of matrices
% 16D70: Structure and classification for modules, bimodules and ideals, direct sum decomposition and cancellation in associative algebras)
% 16E99: Generic for Associative rings and algebras
% 17C27: Idempotents, Peirce decompositions (in nonassociative rings and algebras)
% 18B35: Preorders, orders, domains and lattices
%
% 20-XX: GROUP THEORY AND GENERALIZATIONS
% 20Mxx: **Semigroups**
% 20M13: Arithmetic theory of monoids
% 20M25: Semigroup rings, multiplicative semigroups of rings
% 47A68: Factorization theory of linear operators
%
%
%\thanks{The author was supported by ...}
%
\begin{abstract}
Let $S$ be a numerical monoid, i.e., a submonoid of the additive monoid $(\mathbb N, +)$ of non-neg\-a\-tive integers such that $\mathbb N \setminus S$ is finite. Endowed with the operation of set addition, the family of all finite subsets of $S$ containing $0$ is itself a monoid, which we denote by $\mathcal P_{{\rm fin}, 0}(S)$.

We show that, if $S_1$ and $S_2$ are numerical monoids and $\mathcal P_{{\rm fin}, 0}(S_1)$ is isomorphic to $\mathcal P_{{\rm fin}, 0}(S_2)$, then $S_1 = S_2$. (In fact, we establish a more general result, in which $S_1$ and $S_2$ are allowed to be subsets of the non-negative rational numbers that contain zero and are closed under addition.) This proves a conjecture of Bienvenu and Geroldinger.  
\end{abstract}
\maketitle
\thispagestyle{empty}

\section{Introduction}
\label{sec:intro}

Let $S$ be a semigroup. (We refer to Howie's monograph \cite{Ho95} for generalities on semigroups and monoids; and unless a statement to the contrary is made, we assume throughout that all semigroups are written multiplicatively). Equipped with the (binary) op\-er\-a\-tion of setwise multiplication defined by
$$
(X, Y) \mapsto \{xy \colon x \in X,\, y \in Y\},
$$
the non-empty subsets of $S$ do also form a semigroup, herein denoted by $\mathcal P(S)$ and called the \evid{large power semigroup} of $S$. 
Additionally, the family of all non-empty \emph{finite} subsets of $S$ is a sub\-semi\-group of $\mathcal P(S)$, herein denoted by $\mathcal P_\fin(S)$ and referred to as the \evid{finitary power semigroup} of $S$. 

The systematic investigation of these structures was initiated by Tamura in the late 1960s and continued by semigroup theorists and computer scientists through the 1980s--1990s. By that time, researchers were especially interested in properties of $S$ that
do or do not ascend to $\mathcal P(S)$, as well as in the study of varieties (namely, classes of finite semigroups that are closed under taking homomorphic images, subsemigroups, and finite direct products) generated by $\mathcal P(S)$ as $S$ ranges over a specified family of finite semigroups (see the surveys by Pin \cite{Pin1986, Pin1995} and Almeida \cite{Alm02} for additional background). A key question, known as the \evid{isomorphism problem for power semigroups} and arisen from work of Tamura and Shafer \cite{Tam-Sha1967}, was to prove or disprove that, if $\mathcal P(S)$ is semigroup-isomorphic to the large power semigroup of a semigroup $T$, then $S$ is semigroup-isomorphic to $T$. Answered (in the negative) by Mogiljanskaja \cite{Mog1973}, the question is still open for \emph{finite} semigroups (cf.~the first paragraph on p.~5 of \cite{Ham-Nor2009}). 

Now suppose that $M$ is a monoid with identity $1_M$. Both $\mathcal P(M)$ and $\mathcal P_\fin(M)$ are then monoids too, their identity being the singleton $\{1_M\}$. Accordingly, we call $\mathcal P_\fin(M)$ the \evid{finitary power monoid} (for short, FPM) of $M$. Moreover, the family of all finite subsets of $M$ containing $1_M$ is a submonoid of $\mathcal P_\fin(M)$, herein denoted by $\mathcal P_{\fin,1}(M)$ and named the \evid{reduced FPM} of $M$. Introduced by Fan and Tringali in \cite{Fa-Tr18} (see \cite{Mar-Pin1973} for the analogous notion where finite subsets are replaced with infinite ones) and further investigated in \cite{An-Tr18}, \cite[Sect.~4.2]{Tr20(c)}, and \cite{Bie-Ger-22}, reduced FPMs are the most basic objects in a complex hierarchy of ``highly non-cancellative'' monoids, generically dubbed as \textsf{power monoids}. 

Power monoids provide a natural algebraic framework for a number of problems of great importance in additive combinatorics and related fields, one notable example being S\'ark\"ozy's conjecture \cite[Conjecture 1.6]{Sar2002} on the ``additive irreducibility'' of the set of quadratic residues modulo $p$ for all but finitely many primes $p \in \mathbb N^+$ (the positive integers). 
Furthermore, the arithmetic of power monoids --- with emphasis on ques\-tions concerned with the possibility or impossibility of factoring certain sets as a product of other sets that, in a sense, cannot be ``broken down into smaller pieces'' --- is a rich topic in  itself and has been pivotal in the ongoing development of a ``unifying theory of factorization'' \cite{Cos-Tri2023, Tr20(c)} that extends beyond the (somewhat narrow) boundaries of the classical theory \cite{Ge-HK06}.

The present article adds to this line of research by resolving a conjecture put forth by Bienvenu and Geroldinger \cite{Bie-Ger-22} that shares a similar spirit with Tamura and Shafer's isomorphism problem for power semigroups. Our starting point is the following:

\begin{remark}\label{rem:iso-implies-ISO}
Assume $f \colon H \to K$ is a monoid homomorphism and let $F$ be the function 
$$
\mathcal P_{\fin,1}(H) \to \mathcal P_{\fin,1}(K) \colon X \mapsto f(X), 
$$
where $f(X) := \{f(x) \colon x \in X\} \subseteq K$ is the (direct) image of $X$ under $f$. Note that $F$ is a well-defined func\-tion, because $f$ sends the identity $1_H$ of $H$ to the identity $1_K$ of $K$ and hence $f(X)$ is a finite subset of $K$ containing $1_K$ for each $X \in \mathcal P_{\fin,1}(H)$. In addition, we have $F(\{1_H\}) = \{1_K\}$, namely, $F$ maps the identity of $\mathcal P_{\fin,1}(H)$ to the identity of $\mathcal P_{\fin,1}(K)$; and it is immediate that, for every $X, Y \in \mathcal P_{\fin,1}(H)$,
\begin{equation*}
F(XY) = \{f(xy) \colon x \in X,\, y \in Y\} = \{f(x)f(y) \colon x \in X,\, y \in Y\} = F(X) F(Y).
\end{equation*}
%since $f$ is a semigroup homomorphism $H \to K$.
%where we have especially used that $f(xy) = f(x) f(y)$ for all $x, y \in H$.
%(since $f$ is a monoid homomorphism from $H$ to $K$). 
All in all, this yields that $F$ is a monoid homomorphism $\mathcal P_{\fin,1}(H) \to \mathcal P_{\fin,1}(K)$.

Now suppose that $f$ is bijective. If $X, Y \in \mathcal P_{\fin,1}(H)$ and $X \ne Y$, then there is at least one element $z \in H$ such that either $z \in X \setminus Y$ or $z \in Y \setminus X$. In either case, we gather from the injectivity of $f$ that $F(X) \ne F(Y)$, which ultimately shows that $F$ itself is injective. Moreover, if $Y \in \mathcal P_{\fin,1}(K)$, then $X := \allowbreak f^{-1}(Y)$ is a finite subset of $H$ with $1_H \in X$ and $F(X) = Y$; that is, $F$ is surjective. So, $F$ is a monoid isomorphism from $\mathcal P_{\fin,1}(H) $ to $\mathcal P_{\fin,1}(K)$.
\end{remark}

The examples below prove that, in general, the conclusions drawn in Remark \ref{rem:iso-implies-ISO} cannot be reversed: The reduced FPMs of two monoids $H$ and $K$ can be monoid-isomorphic even if $H$ and $K$ are not.

\begin{examples}
\label{exa:the-unrestricted-conjecture-is-false}
\begin{enumerate*}[label=\textup{(\arabic{*})}, mode=unboxed]
\item\label{exa:(1)}
Let $H$ be a $2$-element monoid. 
Up to isomorphism, $H$ is then either the additive group of integers modulo $2$, or the idempotent submonoid $E = \{0, 1\}$ of the multiplicative monoid of the ring of integers. 
In both cases, $\mathcal P_{\fin,1}(H)$ is an idempotent $2$-element monoid (whose elements are the sets $\{1_H\}$ and $H$), and hence it is monoid-isomorphic to $E$. 
\end{enumerate*}

\vskip 0.05cm

\begin{enumerate*}[label=\textup{(\arabic{*})}, mode=unboxed, resume]
\item\label{exa:(2)} Let $M$ be a \evid{breakable monoid}, meaning that $xy \in \{x, y\}$ for all $x, y \in M$ (see, e.g., Sect.~27 in \cite{Red1967}). It is then easily found that $\mathcal P_{\fin,1}(M)$ is the monoid obtained by endowing the family of all finite subsets of $M$ containing the identity $1_M$ with the operation $(X, Y) \mapsto X \cup Y$. Indeed, we have
$$
X \cup Y = X1_M \cup 1_M Y \subseteq XY = \{xy \colon x \in X, \,y \in Y\} \subseteq X \cup Y,
\qquad \text{for all } X, Y \in \mathcal P_{\fin,1}(M).
$$
\indent{}Now let $H$ be a \evid{unitization of the left zero semigroup} on a set $V$, i.e., define an (associative) multiplication on $V$ by taking $xy := x$ for all $x, y \in V$, add a new element $e$ to $V$, and extend the previous operation from $V$ to $V \cup \{e\}$ by taking $xe = ex := x$ for every $x \in V \cup \{e\}$. Next, let $H^{\rm op}$ be the \evid{opposite monoid} of $H$, so that, in $H^{\rm op}$, we have $xy = y$ for all $x, y \in V$. Of course, $H$ and $H^{\rm op}$ are both breakable monoids (with identity $e$), which implies from the above that $\mathcal P_{\fin,1}(H) = \mathcal P_{\fin,1}(H^{\rm op})$. Yet, $H$ and $H^{\rm op}$ are non-isomorphic monoids unless $|V| \le 1$. \\

\indent{}In fact, assume $|V| \ge 2$ and suppose for a contradiction that there is a monoid isomorphism $f \colon H \to \allowbreak H^{\rm op}$. There then exist $u, v \in V$ with $u \ne v$ such that $uv = v$ in $H^{\rm op}$. Since $f(e) = e$ and $f$ is a bijection, we can thus find $x, y \in V$ with $f(x) = u$ and $f(y) = v$. This is however impossible, since $f$ being a monoid ho\-mo\-mor\-phism $H \to H^{\rm op}$ and considering that $xy = x$ in $H$ yield $v = uv = f(x) f(y) = f(xy) = f(x) = u$.
\end{enumerate*}

\vskip 0.05cm

\begin{enumerate*}[label=\textup{(\arabic{*})}, mode=unboxed, resume]
\item\label{exa:(3)} We get from items \ref{exa:(1)}  and \ref{exa:(2)} that, for any cardinal number $\kappa \ge 2$ (no matter if finite or not), there exist non-isomorphic monoids $H$ and $K$ with $|H| = |K| = \kappa$ such that $\mathcal P_{\fin,1}(H)$ and $ \mathcal P_{\fin,1}(K)$ are monoid-isomorphic. It is open whether anything similar is possible when $H$ and $K$ are \emph{both} cancellative.
\end{enumerate*}
\end{examples}

We are thus led to the question: Are there any \emph{interesting} classes $\mathscr H$ of monoids with the property that $\mathcal P_{\fin,1}(H)$ is (monoid-)isomorphic to $\mathcal P_{\fin,1}(K)$, for some $H, K \in \mathscr H$, if and only if $H$ is i\-so\-mor\-phic to $K$? It has been conjectured by Bienvenu and Geroldinger that the answer is positive when $\mathscr H$ is the class of numerical monoids, where we recall that a \evid{numerical monoid} is a submonoid $S$ of the additive monoid $(\mathbb N, +)$ of the non-negative integers such that $\mathbb N \setminus S$ is a finite set.

More precisely, Conjecture 4.7 in \cite{Bie-Ger-22} states that the reduced FPM of a numerical monoid $S_1$ is isomorphic to the reduced FPM of a numerical monoid $S_2$ if and only if $S_1 = S_2$. (It is folklore that two numerical monoids are isomorphic if and only if they are equal, see \cite[Theorem 3]{Hig-969}.) Our main result (Theorem \ref{thm:Bienvenu-Geroldinger-conjecture}) proves a stronger version of the conjecture, in which numerical monoids are replaced by (\evid{rational}) \evid{Puiseux monoids}, that is, subsets of the non-negative rational numbers that contain zero and are closed under addition \cite{Go17, Ge-Go-Tr21}. 
The corresponding question for the FPMs of numerical monoids was already settled by Bienvenu and Geroldinger themselves in \cite[Theorem 3.2(3)]{Bie-Ger-22}.

\section{The proof}

For all $a, b \in \mathbb Z$ and $\emptyset \ne A \subseteq \mathbb Z$, we take $\llb a, b \rrb := \{x \in \mathbb Z \colon a \le x \le b\}$ to be the \textsf{discrete interval} from $a$ to $b$ and $\gcd A$ to be the \evid{greatest common divisor} of $A$, i.e., the largest non-negative divisor of each $a \in A$ (here the term ``largest'' refers to the divisibility preorder on the multiplicative monoid of the integers and not to the usual ordering of $\mathbb Z$). Observe that $\gcd \{0\} = 0$. 

We denote by $\mathcal P_{\fin,0}(S)$ the reduced FPM of a Puiseux monoid $S$. In contrast to what is done in the introduction with the reduced FPM of an arbitrary monoid $H$, we write $\mathcal P_{\fin,0}(S)$ additively. In particular, given an integer $k \ge 0$ and a positive integer $n$, we let $kX$ be the \evid{$k$-fold sum} and $X/n := \{x/n \colon x \in X\}$ be the \evid{$1/n$-dilate} of a set $X \in \mathcal P_{\fin,0}(\mathbb Q_{\ge 0})$, where $\mathbb Q_{\ge 0}$ is the non-negative cone of the (ordered) field of rational numbers. Further ter\-mi\-nol\-o\-gy and notations, if not explained, are standard or should be clear from the context. Most notably, all morphisms mentioned in this section are \emph{monoid} homomorphisms.

A key ingredient in our proof of the Bienvenu--Geroldinger conjecture is a clas\-si\-cal result of Nathanson \cite{Nat78}, often referred to as the Fundamental Theorem of Additive Combinatorics:

\begin{theorem}[Nathanson's theorem]\label{thm:FTAC}
Let $A$ be a finite subset of $\mathbb N$ with $0 \in A$ and $\gcd A = 1$. There then exist $b, c \in \mathbb N$, $B \subseteq \allowbreak \llb 0, b-2 \rrb$, and $C \subseteq \llb 0, c-2 \rrb$ such that $kA = B \cup \allowbreak \llb b, \allowbreak ka - c \, \rrb \cup \allowbreak (ka - C)$ for every integer $k \ge a^2 n$, where $a := \allowbreak \max A$ and $n := |A| - 1$.
\end{theorem}

The bound $a^2 n$ in Theorem \ref{thm:FTAC} has been recently improved to $a-n+1$ by Granville and Walker \cite[Theorem 1]{Gra-Wal21}, and their result was further refined by Lev in \cite[Theorem 3]{Lev-22} (though nothing of this will be needed in the proofs below).
Also note that, according to our definitions, the condition $\gcd A = 1$ in Theorem \ref{thm:FTAC} implies that the set $A$ has at least two elements --- and it holds that $|A| = 2$ if and only if $A = \{0, 1\}$, in which case $kA = \llb 0, k \rrb$ for each $k \in \mathbb N$ (and hence the theorem is trivial).

\begin{lemma}\label{lem:FTAC-decomposition}
If $A \in \mathcal P_{\fin,0}(\mathbb Q_{\ge 0})$, then $(k+1)A = kA + \{0, \max A\}$ for all large $k \in \mathbb N$.
\end{lemma}

\begin{proof}
Since $A$ is a non-empty finite subset of $\mathbb Q_{\ge 0}$ containing $0$, there exist $d \in \allowbreak \mathbb N^+$ and $A' \in \allowbreak \mathcal P_{\fin,0}(\mathbb N)$ such that $A = A'/d$. It follows that $\max A = \max A'/d$ and $kA = kA'/d$ for all $k \in \mathbb N$, which makes it possible to simplify the remainder of the proof by taking $d = 1$ and hence $A \in \mathcal P_{\fin,0}(\mathbb N)$. 

Accordingly, set $a := \max A$ and $q := \gcd A$. We may suppose $\{0\} \subsetneq A$, or else $kA = \{0\}$ for every $k \in \allowbreak \mathbb N$ and the conclusion is obvious. Then, $q$ is a positive integer and there is no loss of generality in as\-sum\-ing (as we do) that $q = 1$, since letting $\hat{A} := A/q \subseteq \mathbb N$ and $\hat{a} := \max \hat{A}$ yields that (i) $\gcd \hat{A} = \allowbreak 1$ and (ii) $(k+1)A = kA + \{0, a\}$, for some $k \in \mathbb N$, if and only if $(k+1)\hat{A} = k\hat{A} + \{0, \hat{a}\}$.

Thus, we gather from Nathanson's theorem (Theorem \ref{thm:FTAC}) that there exist non-negative integers $b$, $c$, and $k_0$ and sets $B \subseteq \llb 0, b-2 \rrb$ and $C \subseteq \llb 0, c-2\rrb$ such that 
\begin{equation}\label{equ:FTAC-applied}
kA = B \cup \llb b, ka - c \, \rrb \cup (ka - C),
\qquad
\text{for each } k \in \mathbb N_{\ge k_0}.
\end{equation}
Now fix an integer $h \ge \max \{k_0,  1+(b+c)/a\}$. It is obvious that $
hA + \{0, a\} \subseteq hA + A = (h + 1)A$,
and we claim that also the converse inclusion is true. In fact, pick $x \in (h + 1)A$. 

Since $h \ge k_0$, Eq.~\eqref{equ:FTAC-applied} holds for both $k = h$ and $k = h+1$. So, either $x \in B \cup \llb b, ha - c - 1 \rrb$, and it is then evident that $x \in hA$; or $x-a \in \llb (h-1)a - c, ha - c \, \rrb \cup (ha - C) \subseteq hA$ and then $x \in hA + a$, where we are especially using that $h \ge 1+(b+c)/a$ and hence $(h-1)a - c \ge b$. In both cases, the conclusion is that $x \in hA + \{0, a\}$, which finishes the proof as $x$ is an arbitrary element in $(h+1)A$.
\end{proof}

\begin{lemma}
\label{lem:2-element-sets-to-2-elements-sets}
An isomorphism $\phi \colon \mathcal P_{\fin,0}(S_1) \to \mathcal P_{\fin,0}(S_2)$, where $S_1$ and $S_2$ are Puiseux monoids, sends a $2$-element set to a $2$-element set.
\end{lemma}

\begin{proof}
Fix a non-zero $a \in S_1$. We need to show that $B := \allowbreak \phi(\{0, \allowbreak a\}) = \allowbreak \{0, b\}$ for some (non-zero) $b \in S_2$. 
Indeed, set $b := \max B \in S_2$ and note that $b$ is non-zero, because $B \ne \phi(\{0\}) = \{0\}$ (by the injectivity of $\phi$). 
By Lemma \ref{lem:FTAC-decomposition}, there then exists an integer $k \ge 0$ such that 
$
(k+1) B = kB + \{0, b\}$.

Put $A := \phi^{-1}(\{0, b\})$, where $\phi^{-1}$ is the (functional) inverse of $\phi$. Since $\phi^{-1}$ is an isomorphism from $\mathcal P_{\fin,0}(S_2)$ to $\mathcal P_{\fin,0}(S_1)$ with $\phi^{-1}(B) = \{0, a\}$, we get from the above that
$$
(k+1)\{0, a\} = (k+1) \phi^{-1}(B) = k \phi^{-1}(B) + \phi^{-1}(\{0, b\}) = k \{0, a\} + A.
$$
It follows that $\{0\} \subsetneq A \subseteq (k+1) \{0, a\}$ and $\max A = (k+1) a - ka = a$. So, noticing that $a$ is the least non-zero element of $(k+1) \{0, a\}$, we find $A = \{0, \allowbreak a\}$ and hence $B = \allowbreak \phi(\{0, a\}) = \phi(A) = \{0, b\}$.
\end{proof}

\begin{lemma}\label{lem:max-is-homomorphism}
Let $\phi \colon \mathcal P_{\fin,0}(S_1) \to \mathcal P_{\fin,0}(S_2)$ be an isomorphism, where $S_1$ and $S_2$ are Puiseux monoids, and pick $a_1, a_2 \in S_1$. The following hold:
\begin{enumerate}[label=\textup{(\roman{*})}]
\item\label{cor:max-is-homomorphism(i)} There exist $b_1, b_2 \in S_2$ such that $\phi(\{0, a_1\}) = \{0, b_1\}$ and $\phi(\{0, a_2\}) = \{0, b_2\}$.
\item\label{cor:max-is-homomorphism(ii)} $\phi(\{0, a_1 + a_2\}) = \{0, b_1 + b_2\}$.
\end{enumerate}
\end{lemma}

\begin{proof}
Define 
$
A := \{0, a_1\} + \{0, a_2\}$, $
B := \allowbreak \phi(A)$, and $
a_0 := a_1 + a_2 = \max A \in S_1$.
We have from Lemma \ref{lem:2-element-sets-to-2-elements-sets} that, for each $i \in \llb 0, 2 \rrb$, there is an element $b_i \in S_2$ such that 
$
\phi(\{0, \allowbreak a_i\}) = \{0, b_i\}$.
%\phi(\{0, \allowbreak a_1\}) = \{0, b_1\},
%\qquad\text{and}\qquad 
%\phi(\{0, \allowbreak a_2\}) = \{0, b_2\}. 
So, it only remains to show that $b_0 = b_1 + b_2$.
To this end, we know from Lemma \ref{lem:FTAC-decomposition} that $
(k+\allowbreak 1) A = \allowbreak kA + \{0, \allowbreak a_0\}$ for some $k \in \mathbb N$. Since $\phi(X+Y) = \phi(X) + \phi(Y)$ for all $X, Y \in \mathcal P_{\fin,0}(S_1)$, it is thus found that 
$$
B = \phi(A) = \phi(\{0, a_1\}) + \phi(\{0, a_2\}) = \{0, b_1\} + \{0, b_2\}
$$
and
$$
(k+1) B = (k+1) \phi(A) = \phi((k+1) A) = k \phi(A) + \phi(\{0, a_0\}) = kB + \{0, b_0\}.
$$
Consequently, we obtain that $b_0 = (k+1) \max B - k \max B = \max B = b_1 + b_2$, as wished.
\end{proof}

\begin{theorem}\label{thm:Bienvenu-Geroldinger-conjecture}
The reduced finitary power monoids $\mathcal P_{\fin,0}(S_1)$ and $\mathcal P_{\fin,0}(S_2)$ of two Puiseux monoids $S_1$ and $S_2$ are isomorphic if and only if $S_1$ and $S_2$ are.
\end{theorem}

\begin{proof}
In view of Remark \ref{rem:iso-implies-ISO}, we may focus our attention on the ``only if'' part of the statement. 

Let $\phi$ be an isomorphism $\mathcal P_{\fin,0}(S_1) \to \mathcal P_{\fin,0}(S_2)$. By
Lemma \ref{lem:2-element-sets-to-2-elements-sets}, $\phi$ maps a $2$-element set $\{0, a\} \subseteq \allowbreak S_1$ to a $2$-element set $\{0, b\} \subseteq S_2$.
Conversely, any $2$-element set $\{0, b\} \subseteq S_2$ is the image under $\phi$ of a $2$-element set $\{0, a\} \subseteq S_1$, because the inverse $\phi^{-1}$ of $\phi$ is itself an isomorphism, with the result that, for each non-zero $b \in S_2$, there is a non-zero $a \in S_1$ with $\phi^{-1}( \{0, b\} ) = \{0, a\}$.

It follows that the function $\Phi \colon S_1 \to S_2 \colon a \mapsto \max \phi( \{0, a\} )$ is bijective; and on the other hand, we get from Lemma \ref{lem:max-is-homomorphism} that $\Phi$ is a homomorphism (from $S_1$ to $S_2$). So, we are done.
\end{proof}

\begin{corollary}
The reduced finitary power monoids of two numerical monoids $S_1$ and $S_2$ are isomorphic if and only if $S_1 = S_2$.
\end{corollary}

\begin{proof}
Straightforward from Theorem \ref{thm:Bienvenu-Geroldinger-conjecture} and the fact that two numerical monoids are isomorphic if and only if they are equal (see the comments after Examples \ref{exa:the-unrestricted-conjecture-is-false}).
\end{proof}

\section*{Acknowledgements}

The authors are thankful to Benjamin Steinberg (City University of New York, US) for pointing out Higgins' paper \cite{Hig-969} on MathOverflow (see \url{https://mathoverflow.net/questions/437667/}), and to the anonymous referees for their careful reading of the paper and many comments that have greatly helped to improve on the presentation.

\end{document}